\documentclass[12pt,english]{amsart}
\usepackage{babel}

 \theoremstyle{plain}
\newtheorem{thm}{Theorem}[section]
  \theoremstyle{definition}
  \newtheorem{defn}[thm]{Definition}
 \theoremstyle{definition}
  \newtheorem{example}[thm]{Example}
  \theoremstyle{plain}
  \newtheorem{lem}[thm]{Lemma}
  \theoremstyle{plain}
  \newtheorem{cor}[thm]{Corollary}
  \theoremstyle{plain}
  \newtheorem{prop}[thm]{Proposition}
  \theoremstyle{remark}
  
\theoremstyle{plain}
	\newtheorem{question}[thm]{Question}

\usepackage[all,dvips]{xy}

\usepackage{mathbbol}

\begin{document}

\title{Weakly Projective C*-Algebras}

\author{Terry A. Loring}

\address{Department of Mathematics and Statistics, University of New Mexico,
Albuquerque, NM 87131, USA.}




\markleft {Weakly Projective C*-Algebras (preprint version)}
\markright {Weakly Projective C*-Algebras (preprint version)}

\begin{abstract}
The noncommutative analog of an approximative absolute retract (AAR)
is introduced, a weakly projective $C^{*}$-algebra. This property
sits between being residually finite dimensional and projectivity.
Examples and closure properties are considered.
\end{abstract}

\maketitle

\section{Introduction }

The noncommutative analogs of absolute retracts and absolute neighborhood
retracts in the category of $C^{*}$-algebras are the 
projective (\cite{EffrosKaminkerShape})
and semiprojective (\cite{Blackadar-shape-theory}) $C^{*}$-algebras.
In applications, semiprojectivity is often not the most desirable
property; many authors have looked instead at weak semiprojectivity
(\cite{EilersLoringContingenciesStableRelations}). For example, see
\cite{DadarlatElliottFieldsOfKirchberg,LinWeakSemiprojectivity,
SpielbergWeakSemiprojectivity,HadwinLiApproxLift}.

Using what are called approximative retracts, Clapp, many years ago
in \cite{ClappAANR}, defined approximative absolute retracts (AAR)
and approximative absolute neighborhood retracts (AANR). The relation
between AANR spaces and weakly semiprojective $C^{*}$-algebras will be
explored elsewhere. Here, we get started on a noncommutative analog
of AAR, the weakly projective $C^{*}$-algebra.

The class of weakly projective $C^{*}$-algebras has some of the expected
closure properties. In addition, weak projectivity for $A$ is enough
to imply that $A$ is residually finite dimensional.

In \cite{chigogidzeNonComARs} it has been determined which compacta
$X$ have $C_{0}(X\setminus\{x_{0}\})$ projective---the dendrites.
It would be nice to know when $C_{0}(X\setminus\{x_{0}\})$ is semiprojective,
weakly projective or weakly semiprojective. 

The reader is warned that what is called weak
projectivity in \cite{PedersenExtensions} is weak semiprojectivity.

Many of the ideas here were inspired by ongoing collaborations with
S\o ren Eilers and Tatiana Shulman.

There are potentially more definitions and results related to absolute
neighborhood retracts than will be interesting
when adapted to $C^{*}$-algebras.  Some places these might be
found are \cite{VanMillInfDimTop} and the more
classic \cite{BorsukTheoryOfShape} and \cite{HuRetracts}.
For $C^{*}$-algebras recently found to be projective,
see \cite{chigogidzeNonComARs,LoringProjectiveKtheory,Shulman_nilpotents}.

\section{Approximative Absolute Retracts (AARs)}

In defining approximative absolute retracts we follow \cite{ClappAANR}.
Recall that a \emph{compactum} is a compact, metrizable space. 

\begin{defn}
A compactum $X$ is an \emph{approximative absolute retract (AAR)}
if, whenever $X$ is a closed subset of a compactum $Y,$ there is
a sequence $r_{n}$ of continuous functions $r_{n}:Y\rightarrow X$
so that
\[
\lim_{n\rightarrow\infty}r_{n}(x)=x
\]
uniformly over $x$ in $X.$
\end{defn}

We next use a pushout to get an approximate extension property. This
is a variation on an old trick. See \cite[Proposition 3.2]{HuRetracts}.

\begin{thm}
\label{thm:AARbyApproxExtension}
Let $X$ be a campactum. Then $X$
is an AAR if, and only if, whenever $Z$ is a closed subset of a compactum
$Y$ and $f:Z\rightarrow X$ is continuous, there is a sequence $g_{n}$
of continuous function $g_{n}:Y\rightarrow X$ for which 
$g_{n}(z)\rightarrow f(z)$
uniformly over $z$ in $Z.$ To summarize in a diagram: 
\[
\xymatrix{
 & Y \ar@{-->}[dl] _{ g_n }\\
 X & Z \ar[l] ^{f} \ar@{^{(}->}[u]
}
\]
\end{thm}

\begin{proof}
Suppose $X$ is an AAR and we are given $Y,$ $Z,$ and $f$ as indicated.
Take the pushout, or adjunction space: 
\[
\xymatrix{
 X \cup_Z Y  & Y \ar[l] _(0.3){ \iota_2 }  \\
 X     \ar[u] ^ { \iota_1 }     
 		& Z  \ar[l] ^{ f } \ar@{^{(}->}[u] 
}  
\]
Notice that
$X\cup_{Z}Y$ is a compact metrizable space and that $\iota_{1}$
is an inclusion. We can apply the definition of AAR and find 
\[
\overline{r}_{n}:X\cup_{Z}Y\rightarrow X
\]
with $\overline{r}_{n}\circ\iota_{1}(w)\rightarrow w$ uniformly.
Therefore, when $z$ is in $Z,$
\[
\lim_{n}\overline{r}_{n}\circ\iota_{2}(z)
= \lim_{n}\overline{r}_{n}\circ\iota_{1}(f(z))
= f(z)
\]
uniformly, so we may set $g_{n}=\overline{r}_{n}\circ\iota_{2}.$

To prove the converse, assume the second condition holds and that
$X$ is a closed subset of a compactum $Y.$ We can find $g_{n}$
as in this diagram 
\[
\xymatrix{
 & Y \ar@{-->}[dl]_{g_n}\\
 X & X \ar[l] ^(0.4){\mathrm{id}_X } \ar@{^{(}->}[u] 
}
\]
with $g_{n}(x)\rightarrow\mbox{id}_{X}(x)$
uniformly for $x$ in $X.$ We set $r_{n}=g_{n}.$ 
\end{proof}

\begin{cor}
\label{cor:C*meaningOfAAR} 
Suppose $X$ is a compactum. Then $X$
is an AAR if, and only if, for every unital surjection 
$\pi:B\rightarrow C$
between separable, unital, commutative $C^{*}$-algebras, and for
every unital $*$-homomorphism $\varphi:C(X)\rightarrow C,$ there
is a sequence $\varphi_{n}:C(X)\rightarrow B$ of 
unital $*$-homomorphisms
so that $\pi\circ\varphi_{n}\rightarrow\varphi.$
\end{cor}

\begin{proof}
This is straightforward, except perhaps the meaning of the convergence.
We require
\[
\lim_{n\rightarrow\infty}
\left\Vert \pi\circ\varphi_{n}(h)-\varphi(h)\right\Vert =0
\]
 for each $h$ in $C(X).$ 
\end{proof}

Of course, every AR is an AAR. To see examples of AARs that are not
AR, we can use the following, a rewording 
of \cite[Theorem 2.3]{ClappAANR}.

\begin{thm}
\label{thm:approxByARisAAR}
Suppose $X$ is a compactum and that $\theta_{n}:X\rightarrow X$
is a sequence of continuous functions that converges uniformly to
the identity. If each $\theta_{n}(X)$ is an AAR then $X$ is an AAR.
\end{thm}

\begin{proof}
Let $d$ be a compatible metric on $X.$ Passing to a subsequence
we may assume 
\[
d(\theta_{n}(x),x)\leq\frac{1}{n}
\]
for all $n$ and all $x.$ Suppose $X$ is a closed subset of a compactum
$Y.$ We apply 
Theorem~\ref{thm:AARbyApproxExtension} to $\theta(X)$
to find continuous $r_{n}$ as in this diagram, 
\[
\xymatrix{
	&	& Y  \ar[dl] _{r_n}
\\
X 
	& \theta_n (X) \ar@{_{(}->}[l]
		& X \ar@{^{(}->}[u] \ar[l] ^(0.4){\theta_n}
}
\]
with
\[
d(r_{n}(x),x)\leq\frac{1}{n}
\]
for all $x$ in $X.$ Therefore
\[
d(r_{n}(x),x)
\leq d(r_{n}(x),\theta_{n}(x))+d(\theta_{n}(x),x)
\leq\frac{2}{n}
\]
for all $x$ in $X.$
\end{proof}

\begin{example}
\label{example:topSinceCurve}
(\cite[Example 2.2]{ClappAANR})
For an AAR that is not an AR, we have the topologist's sine curve
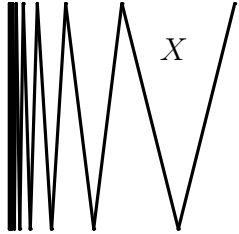
\[
\xygraph{
!~-{@{-}@[|(3.0)]}
!{<0cm,0cm>;<3.0cm,0.0cm>:<0.0cm,3.0cm>::}  
!{(0.73,0.8)}*{X}="X"
!{(0,1)}*{}="01"
!{(0,0)}*{}="00"
!{(1,1)}*{}="a1"
!{(0.75,0)}*{}="b1"
!{(0.5,1)}*{}="a2"
!{(0.375,0)}*{}="b2"
!{(0.25,1)}*{}="a3"
!{(0.1875,0)}*{}="b3"
!{(0.125,1)}*{}="a4"
!{(0.09375,0)}*{}="b4"
!{(0.0625,1)}*{}="a5"
!{(0.046875,0)}*{}="b5"
!{(0.03125,1)}*{}="a6"
!{(0.0234375,0)}*{}="b6"
!{(0.015625,1)}*{}="a7"
!{(0.01171875,0)}*{}="b7"
!{(0.0078125,1)}*{}="a8"
!{(0.005859375,0)}*{}="b8"
!{(0.00390625,1)}*{}="a9"
!{(0.0029296875,0)}*{}="b9"
!{(0.001953125,1)}*{}="a10"
!{(0.0014648438,0)}*{}="b10"
"00"-"01"
"a1"-"b1"
"b1"-"a2"
"a2"-"b2"
"b2"-"a3"
"a3"-"b3"
"b3"-"a4"
"a4"-"b4"
"b4"-"a5"
"a5"-"b5"
"b5"-"a6"
"a6"-"b6"
"b6"-"a7"
"a7"-"b7"
"b7"-"a8"
"a8"-"b8"
"b8"-"a9"
"a9"-"b9"
"b9"-"a10"
"a10"-"b10"
}
\]
There is an increasing sequence of closed subsets $X_{n}$ with dense
union where each $X_{n}$ is homeomorphic to a closed interval. 
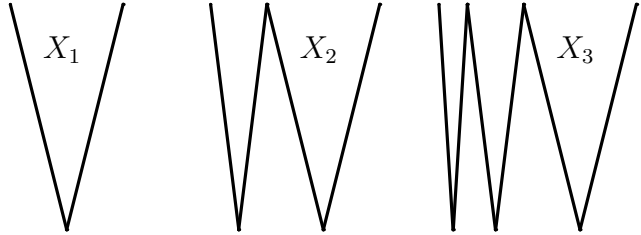
\[
\xygraph{
!~-{@{-}@[|(3.0)]}
!{<0cm,0cm>;<3.0cm,0.0cm>:<0.0cm,3.0cm>::}  
!{(0.73,0.8)}*{X_1}="X"
!{(0,1)}*{}="01"
!{(0,0)}*{}="00"
!{(1,1)}*{}="a1"
!{(0.75,0)}*{}="b1"
!{(0.5,1)}*{}="a2"
"a1"-"b1"
"b1"-"a2"
}
\quad
\xygraph{
!~-{@{-}@[|(3.0)]}
!{<0cm,0cm>;<3.0cm,0.0cm>:<0.0cm,3.0cm>::}  
!{(0.73,0.8)}*{X_2}="X"
!{(0,1)}*{}="01"
!{(0,0)}*{}="00"
!{(1,1)}*{}="a1"
!{(0.75,0)}*{}="b1"
!{(0.5,1)}*{}="a2"
!{(0.375,0)}*{}="b2"
!{(0.25,1)}*{}="a3"
"a1"-"b1"
"b1"-"a2"
"a2"-"b2"
"b2"-"a3"
}
\quad
\xygraph{
!~-{@{-}@[|(3.0)]}
!{<0cm,0cm>;<3.0cm,0.0cm>:<0.0cm,3.0cm>::}  
!{(0.73,0.8)}*{X_3}="X"
!{(0,1)}*{}="01"
!{(0,0)}*{}="00"
!{(1,1)}*{}="a1"
!{(0.75,0)}*{}="b1"
!{(0.5,1)}*{}="a2"
!{(0.375,0)}*{}="b2"
!{(0.25,1)}*{}="a3"
!{(0.1875,0)}*{}="b3"
!{(0.125,1)}*{}="a4"
"a1"-"b1"
"b1"-"a2"
"a2"-"b2"
"b2"-"a3"
"a3"-"b3"
"b3"-"a4"
}
\]
The map $r_{n}:X\rightarrow X_{n}$ that sends $X\setminus X_{n}$
horizontally to the left-most ascending segment in $X_{n}$, while
fixing $X_{n},$ gives us 
\[
d(r_{n}(x),x)\leq2^{-n+1}
\]
and so $X$ is an AAR. On the other hand, $X$ is not path connected
and so not an AR.
\end{example}

\section{Pointed Approximative Absolute Retracts}

From the point of view of $C^{*}$-algebras, we need not only $C(X)$
for $X$ a compactum, but most importantly also the ideals $C_{0}(U)$
for open subsets $U.$ We could consider locally compact spaces, but
instead opt to look at pointed compacta. In terms of $C^{*}$-algebras,
a pointed space translates to the surjection $\delta_{\infty}$ in
the exact sequence 
\[
\xymatrix{
0 \ar[r]
	& C_{0}\left(X\right) \ar[r]
		& C\left(X^{+}\right) \ar[r]^(0.65){\delta_{\infty}}
			& \mathbb{C} \ar[r]
				& 0 .
}
\]
In the noncommutative case we
will of course look at $\lambda$ in the exact sequence
\[
\xymatrix{
0 \ar[r]
	& A \ar[r]
		& {\widetilde {A}} \ar[r]^{\lambda}
			& \mathbb{C} \ar[r]
				& 0 .
}
\]
We use $\widetilde{A}$ to mean ``add a unit, no matter what.'' For
a locally compact space $X,$ we use $X^{+}$ to denote the one-point
compactification. If $X$ is compact, then $X^{+}$ has an extra,
isolated point.

Certainly the concepts of AAR and AANR have been explored in the locally
compact setting, as for example in \cite{PowersM_AANRs}. It is basically
a matter of convenience to look instead at pointed compact spaces.
This was the approach taken by Blackadar looking at projectivity and
semiprojectivity in \cite{Blackadar-shape-theory}.

\begin{defn}
\label{def:pointed-AAR} 
A pointed compactum $(X,x_{0})$ is a 
\emph{pointed approximative absolute retract} if,
whenever $X$ is a closed subset
of a compactum $Y,$ there is a sequence $r_{n}$ of
continuous functions
$r_{n}:Y\rightarrow X$ so that
\[
r_{n}(x_{0})=x_{0}
\]
for all $n$ and
\[
\lim_{n\rightarrow\infty}r_{n}(x)=x
\]
uniformly over $x$ in $X.$
\end{defn}

\begin{lem}
Suppose $x_{0}$ is any point in a compactum $X_{0}.$ If $(X,x_{0})$
is a pointed approximative absolute retract then $X$ is an approximative
absolute retract.
\end{lem}

\begin{proof}
Ignore $x_{0}.$
\end{proof}

\begin{example}
\label{exa:topSineCurveBadPoint} 
If $X$ is the topologist's sine
curve, and if $x_{1}$ is the point on the bottom-left of $X$ as
drawn in Example~\ref{example:topSinceCurve}, then $(X,x_{1})$
is not a pointed AAR. 
\end{example}

\begin{proof}
By definition $X$ sits as a closed subset of the unit square $S.$
For $(X,x_{1})$ to be an AAR, we would need $r_{n}:S\rightarrow X$
that fix $x_{1}$ and that come close to fixing elements of $X.$
The points in $X$ off the left edge are not path connected in $X$
to $x_{1}$ and the continuity of $r_{n}$ forces $r_{n}(S)$ to
be a subset of that left edge. This is a contradiction. 
\end{proof}

\begin{thm}
\label{thm:AARbyApproxExtension-pointed}Let $X$ be a campactum and
$x_{0}$ a point in $X.$ Then $(X,x_{0})$ is a pointed AAR if and
only if, whenever $Z$ is a closed subset of a compactum $Y,$ and
$z_{0}$ is a point in $Z$ and $f:Z\rightarrow X$ is continuous
with $f(z_{0})=x_{0},$ there is a sequence $g_{n}$ of continuous
functions $g_{n}:Y\rightarrow X$ for which $g_{n}(z_{0})=x_{0}$
for all $n$ and $g_{n}(z)\rightarrow z$ uniformly for $z$ in $Z.$
\end{thm}

\begin{proof}
The proof of Theorem~\ref{thm:AARbyApproxExtension} can be modified
as follows. In the adjunction space,
\[
\iota_{2}(z_{0})=\iota_{1}(f(z_{0}))=\iota_{1}(x_{0}).
\]
The $\overline{r}_{n}$ can now be found with the additional property
$\overline{r}_{n}(\iota_{1}(x_{0}))=x_{0}$ and so we find
\[
g_{n}(z_{0})
=\overline{r}_{n}(\iota_{2}(z_{0}))
=\overline{r}_{n}(\iota_{1}(x_{0}))
=x_{0}.
\]
\end{proof}

\begin{cor}
\label{cor:C*meaningOfAAR-pointed}
Suppose $X$ is a compactum and
$x_{0}$ is in $X.$ Then $(X,x_{0})$ is a pointed AAR if, and only
if, for every unital surjection $\pi:B\rightarrow C$ between separable,
commutative $C^{*}$-algebras, and for every $*$-homomorphism 
\[
\varphi:C_{0}(X\setminus\{x_{0}\})\rightarrow C,
\]
there is a sequence 
 \[
\varphi_{n}:C_{0}(X\setminus\{x_{0}\})\rightarrow B
\]
 of $*$-homomorphisms so that $\pi\circ\varphi_{n}\rightarrow\varphi.$
\end{cor}

\begin{proof}
For locally compact spaces $\Lambda$ and $\Omega,$ the pointed continuous
maps from $(\Omega^{+},\infty)$ to $(\Lambda^{+},\infty)$ are in
one-to-one correspondence with the $*$-homorphisms from $C_{0}(\Lambda)$
to $C_{0}(\Omega).$ The $*$-homomorphism $h\mapsto h\circ f$ will
be a surjection if and only if $f:\Omega^{+}\rightarrow\Lambda^{+}$
is injective. Convergence in 
$\hom\left(C_{0}(\Lambda),C_{0}(\Omega)\right)$
corresponds to uniform convergence of functions that preserve the
points at infinity. The result follows.
\end{proof}

\begin{thm}
\label{thm:approxByARisAAR-pointed} 
Suppose $X$ is a compactum,
that $\theta_{n}:X\rightarrow X$ is a sequence of continuous functions
that converges uniformly to the identity and that $x_{0}$ is a point
in $X$ that is fixed by all the $\theta_{n}.$ 
If each $\left(\theta_{n}(X),x_{0}\right)$
is a pointed AAR then $(X,x_{0})$ is a pointed AAR.
\end{thm}

\begin{proof}
Just observe that in the proof of Theorem~\ref{thm:approxByARisAAR},
the $r_{n}$ can now be found fixing $x_{0}.$
\end{proof}

\begin{example}
\label{exa:topSineCurvegoodPoint} 
If $X$ is the topologist's sine
curve, and if $x_{0}$ is the point on the top-right of $X$ as drawn
in Example~\ref{example:topSinceCurve}, then $(X,x_{1})$ is a pointed
AAR. 
\end{example}

\section{A Noncommutative Analog of AAR}

From Corollary~\ref{cor:C*meaningOfAAR-pointed} we see how to define
weak projectivity. In light of Examples~\ref{exa:topSineCurveBadPoint}
and \ref{exa:topSineCurvegoodPoint} we will need to take care when
dealing with unital $C^{*}$-algebras. We will, in fact, never define
a notion of ``weakly projective in the unital category'' but will
define, for not-necessarily-unital $C^{*}$-algebras, the notion of
``weakly projective relative to unital $C^{*}$-algebras.'' This rather
ruins the analogy with the topology, but is more in keeping with how
$C^{*}$-algebraists work. More than zero of us avoid the unital category
for the simple reason that it does not allow for ideals.

\begin{defn}
Suppose $A$ is a separable $C^{*}$-algebra. We say $A$ is \emph{weakly
projective} if, for every $*$-homomorphism $\varphi:A\rightarrow C$
and every surjection $\rho:B\rightarrow C$ of arbitrary $C^{*}$-algebras,
there is a sequence $\varphi_{n}:A\rightarrow B$ of $*$-homomorphisms
so that $\rho\circ\varphi_{n}\rightarrow\varphi.$
\end{defn}

By restricting what surjections $\rho$ is allowed to be, we get weaker
properties.

\begin{defn}
Suppose $A$ is a separable $C^{*}$-algebra. We say $A$ is 
\emph{weakly projective with respect to unital $C^{*}$-algebras} if,
for every
$*$-homomorphism $\varphi:A\rightarrow C$ and every unital surjection
$\rho:B\rightarrow C$ between unital $C^{*}$-algebras, there is
a sequence $\varphi_{n}:A\rightarrow B$ of $*$-homomorphisms so
that $\rho\circ\varphi_{n}\rightarrow\varphi.$
\end{defn}

Obviously projective implies weakly projective and weakly projective
implies weakly projective w.r.t.~unital $C^{*}$-algebras.

\begin{lem}
If $A$ is weakly projective w.r.t.\ unital $C^{*}$-algebras then
$A$ does not have a unit. 
\end{lem}

\begin{proof}
Suppose $A$ is unital. Consider the interval over $A,$
\[
\mathbf{I}A=C\left([0,1],A\right),
\]
and the surjection found by evaluation at both endpoints,
\[
\delta_{0}\oplus\delta_{1}:\mathbf{I}A\rightarrow A\oplus A.
\]
The $*$-homomorphism $\iota_{1}:A\rightarrow A\oplus A$ defined
by $\iota_{1}(a)=(0,a)$ should lift approximately to 
$\psi_{n}:A\rightarrow\mathbf{I}A.$
At $0,$ $\psi_{n}(1)$ will be a projection near $0,$ and so indeed
$\psi_{n}(1)(0)=0$ for large $n.$ The only thing homotopic to $0$
in the space of projections in $A$ is $0$ itself, so we conclude
$\psi_{n}(1)=0$ for large $n.$ Therefore
\[
\left(\left(\delta_{0}\oplus\delta_{1}\right)\circ\psi_{n}\right)(1)
=(0,0)
\]
will not converge to $\iota_{1}(1)=(0,1).$ 
\end{proof}

\begin{thm}
\label{thm:checkWPonlySeparableCase}
If $A$ is a separable $C^{*}$-algebra
then the following are equivalent:
\begin{enumerate}
\item $A$ is weakly projective;
\item for all separable $C^{*}$-algebras $B$ 
and $C,$ and for every $*$-homomorphism
$\varphi:A\rightarrow C$ and every surjection $\rho:B\rightarrow C,$
there is a sequence $\varphi_{n}:A\rightarrow B$ of $*$-homomorphisms
so that $\rho\circ\varphi_{n}\rightarrow\varphi.$
\end{enumerate}
\end{thm}

\begin{proof}
Certainly (a) implies (b). For the reverse, suppose (b) holds and
$\varphi:A\rightarrow B/I$ is given. Let $a_{1},a_{2},\ldots$ be
dense in $A.$ Pick any $b_{1},b_{2},\ldots$ so that $\pi(b_{k})=\varphi(a_{k})$
and let $B_{0}$ denote the $C^{*}$-algebra generated by the $b_{k}.$
This is separable. Let $I_{0}=B\cap I.$ If we let $\varphi_{0}$
denote $\varphi$ but with codomain $B_{0}/I_{0},$ we have the commutative
diagram 
\[
\xymatrix{
 	& B_0   \ar[d] ^{\pi_0}  \ar[r] ^{\iota} 
  		& B \ar[d] ^{\pi}\\
A \ar[r] ^(0.4){\varphi_0} \ar@/_0.3cm/[rr] _{ \varphi}
	& B_0/I_0 \ar[r] ^{\iota} 
		& B/I
}
\]
We know there are $\varphi_{n}:A\rightarrow B_{0}$
with $\pi_{0}\circ\varphi_{n}(a)\rightarrow\varphi_{0}(a),$ and so
$\iota\circ\varphi_{n}$ are the desired approximate lifts.
\end{proof}

\begin{cor}
Suppose $X$ is a locally compact, metrizable space. If $C_{0}(X)$
is weakly projective then $\left(X^{+},\infty\right)$ is a pointed
AAR .\end{cor}
\begin{example}
\label{exa:topSineCurvebadPoint-NC}If $X$ is the topologist's sine
curve, and if $x_{1}$ is the point on the bottom-left of $X$ as
drawn in Example~\ref{example:topSinceCurve}, 
then $C_{0}\left(X\setminus\left\{ x_{1}\right\} \right)$
is not weakly projective.
\end{example}

\begin{thm}
\label{thm:WPfromPapproximations}
Suppose $A$ is a separable $C^{*}$-algebra
and that $\theta_{n}:A\rightarrow A$ is a sequence of $*$-homomorphisms
that converges to the identity map. If each $\theta_{n}(A)$ is weakly
projective then $A$ is weakly projective. If each $\theta_{n}(A)$
is weakly projective w.r.t.\ unital $C^{*}$-algebras then $A$ is
weakly projective w.r.t.~unital $C^{*}$-algebras. 
\end{thm}

\begin{proof}
Assume the $\theta_{n}(A)$ are weakly projective. 
Suppose $\rho:B\rightarrow C$
is a surjection of $C^{*}$-algebras and we are given
also a $*$-homomorphism $\varphi:A\rightarrow C.$
If $a_{1},a_{2},\ldots$ is a dense sequence
in $A$ then we can pass to a subsequence of the $\theta_{n}$ so
that
\[
\left\Vert \theta_{n}(a_{j})-a_{j}\right\Vert \leq \frac{1}{n}
\quad(1\leq j\leq n).
\]
We are now looking at
\[
\xymatrix{
  &&& B \ar[d] ^{\rho} \\
 A \ar[r] _(0.4){\theta_n } 
 	& *+<0.4cm,0.1cm>{\theta_{n}(A)} \ar@{^{(}->}[r]  
		& A \ar[r] _{\varphi}  
			& C
}
\]
Since $\theta_{n}(A)$ is
weakly projective there are $*$-homomorphisms $\varphi_{n}$ as in
this diagram
\[
\xymatrix{
  &&& B \ar[d] ^{\rho} \\
 A \ar[r] _(0.4){\theta_n }
 	& *+<0.4cm,0.1cm>{\theta_{n}(A)} \ar@{^{(}->}[r]  \ar[rru] |(0.16){\hole}^{\varphi_n}
 		& A \ar[r] _{\varphi}  & C
}
\]
with
\[
\left\Vert 
\rho\circ\varphi_{n}(\theta_{n}(a_{j}))-\varphi(\theta_{n}(a_{j}))
\right\Vert 
\leq\frac{1}{n}
\quad(1\leq j\leq n).
\]
Then
\begin{eqnarray*}
 &  &
 \left\Vert 
 \rho\circ\varphi_{n}\circ\theta_{n}(a_{j})-\varphi(a_{j})
 \right\Vert \\
 & \leq &
 \left\Vert
 \rho\circ\varphi_{n}(\theta_{n}(a_{j}))-\varphi(\theta_{n}(a_{j}))
 \right\Vert +
 \left\Vert 
 \varphi(\theta_{n}(a_{j})-a_{j})
 \right\Vert \\
 & \leq &
 \left\Vert
 \rho\circ\varphi_{n}(\theta_{n}(a_{j}))-\varphi(\theta_{n}(a_{j}))
 \right\Vert 
 +
 \left\Vert
 \theta_{n}(a_{j})-a_{j}
 \right\Vert \\
 & \leq &
 \frac{2}{n}
 \end{eqnarray*}
and so the $\varphi_{n}\circ\theta_{n}$ are the desired approximate
lifts.

The proof of the second statement is nearly identical, starting with
the extra assumptions that $B,$ $C$ and $\rho$ are unital.
\end{proof}
While $Y^{+}$ being an absolute retract does not generally lead to
$C_{0}(Y)$ being projective, we do know that $C_{0}(0,1]$ is projective.
This is enough to get the following example. One could get more exotic
examples by starting with more exotic projective $C^{*}$-algebras
as seen, for example, in \cite{LorPederProjTrans}.

\begin{example}
\label{exa:topSineCurvegoodPoint-NC} 
If $X$ is the topologist's
sine curve, and if $x_{0}$ is the point on the top-right of $X$
as drawn in Example~\ref{example:topSinceCurve}, then 
$C_{0}\left(X\setminus\left\{ x_{0}\right\} \right)$
is weakly projective. 
\end{example}

Examples~\ref{exa:topSineCurvebadPoint-NC} 
and \ref{exa:topSineCurvegoodPoint-NC}
show that it is possible to have $\widetilde{A}\cong\widetilde{B}$
with $A$ weakly projective and $B$ not weakly projective.

\begin{thm}
\label{thm:WPforUnitalForDescriptions}
If $A$ is a separable $C^{*}$-algebra
then the following are equivalent:
\begin{enumerate}
\item $A$ is weakly projective w.r.t.\ unital $C^{*}$-algebras;
\item for all separable, unital $C^{*}$-algebras $B$ and $C,$ and for
every $*$-homomorphism $\varphi:A\rightarrow C$ and every unital
surjection $\rho:B\rightarrow C,$ there is a 
sequence $\varphi_{n}:A\rightarrow B$
of $*$-homomorphisms so that $\rho\circ\varphi_{n}\rightarrow\varphi;$
\item for all unital $C^{*}$-algebras $B$ and $C,$ and for every unital
$*$-homomorphism $\varphi:\widetilde{A}\rightarrow C$ and every
unital surjection $\rho:B\rightarrow C,$ there is a 
sequence $\varphi_{n}:\widetilde{A}\rightarrow B$
of unital $*$-homomorphisms so that $\rho\circ\varphi_{n}\rightarrow\varphi.$
\item for all separable, unital $C^{*}$-algebras $B$ and $C,$ and for
every unital $*$-homomorphism $\varphi:\widetilde{A}\rightarrow C$
and every unital surjection $\rho:B\rightarrow C,$ there is a sequence
$\varphi_{n}:\widetilde{A}\rightarrow B$ of unital $*$-homomorphisms
so that $\rho\circ\varphi_{n}\rightarrow\varphi.$
\end{enumerate}
\end{thm}

\begin{proof}
The proof of Theorem~\ref{thm:checkWPonlySeparableCase} works to
show the equivalence of (a) and (b) so long as we set $B_{0}$ to
be the $C^{*}$-subalgebra generated by the $b_{k}$ and $1_{B}.$
Just as easily, we get the equivalence of (c) and (d)

Assume (a), and suppose we are given $B$ and $C$ unital and separable,
$\rho:B\rightarrow C$ a unital surjection 
and $\varphi:\widetilde{A}\rightarrow C$
unital. The assumption on $A$ give us the $\psi_{n}$ in this diagram,
\[
\xymatrix{
	&	& B \ar[d]^{\rho}\\
*+<0.4cm,0.1cm>{A} \ar@{^{(}->}[r] \ar@{-->}[rru] |(0.16){\hole}^{\psi_n }
	&{\widetilde {A}} \ar[r] ^{\varphi} 
		& C
}
\]
with $\rho\circ\psi_{n}(a)\rightarrow\varphi(a)$
for all $a$ in $A.$ We can extend $\psi_{n}$ to a unital $*$-homomorphism
$\varphi_{n}$ on $\widetilde{A}$ by 
\[
\varphi_{n}(a+\alpha\mathbb{1})=\psi_{n}(a)+\alpha1_{B}.
\]
Then
\begin{align*}
\rho\circ\varphi_{n}(a+\alpha\mathbb{1}) 
 & =\rho(\psi_{n}(a)+\alpha1_{B})\\
 & =\rho(\psi_{n}(a))+\alpha\varphi(\mathbb{1})\\
 & \rightarrow\varphi(a)+\alpha\varphi(\mathbb{1})\\
 & =\varphi(a+\alpha\mathbb{1}),
\end{align*}
and we have verified (c).

Assume (c), and suppose $B$ and $C$ are separable and unital and
we are given a $*$-homomorphism $\varphi:A\rightarrow C$ and a unital
surjection $\rho:B\rightarrow C.$ We can extend $\varphi$ to a unital
$\overline{\varphi}$ by 
\[
\overline{\varphi}(a+\alpha\mathbb{1})=\varphi(a)+\alpha1_{C}.
\]
 The assumption on $A$ now gives us the unital $*$-homomorphisms
$\psi_{n}$ in this diagram, 
\[
\xymatrix{
	&	&B \ar[d] ^{\rho} \\
*+<0.4cm,0.1cm>{A} \ar@{^{(}->}[r] \ar@/_0.4cm/[rr] _{\varphi}
	&{\widetilde {A}}  \ar[r] ^{\overline {\varphi}} \ar@{-->}[ru] ^{\psi_n}
		& C
}
\]
with 
\[
\rho\circ\psi_{n}(a+\alpha\mathbb{1})
\rightarrow
\overline{\varphi}(a+\alpha\mathbb{1}).
\]
We take for the needed approximate lifts the restriction
of the $\psi_{n}$ to $A.$ We have verified (a).
\end{proof}

\begin{cor}
Suppose $X$ is a locally compact, metrizable space. If $C_{0}(X)$
is weakly projective w.r.t.\ unital $C^{*}$-algebras then $X^{+}$
is an AAR .
\end{cor}

\begin{cor}
Suppose $A$ and $B$ are separable $C^{*}$-algebras. If $A$ weakly
projective w.r.t.~unital $C^{*}$-algebras
and $\widetilde{A}\cong\widetilde{B}$
then $B$ is weakly projective w.r.t.~unital $C^{*}$-algebras.
\end{cor}

We present the analogs of Theorems~\ref{thm:AARbyApproxExtension}
and \ref{thm:AARbyApproxExtension-pointed}. We also include analogs
of the fact that if $X$ is a compact subset of $[0,1]^{n}$ that
to prove $X$ as an AAR, it suffices to show $X$ is an approximate
retract of $[0,1]^{n}.$ There is a similar statement involving the
Hilbert cube.

The replacement for $[0,1]^{n}$ is a projective $C^{*}$-algebra,
such as the universal $C^{*}$-algebra generated by $n$-contractions.
Such an object is an acquired taste, so we state our result to allow
for a choice of projective $C^{*}$-algebra. The point is that to
test a given $A$ it suffices to work with a single map onto $A$
from a single projective.

\begin{thm}
Suppose $A$ is a separable $C^{*}$-algebra. Each of the following
two conditions is equivalent to $A$ being weakly projective:
\begin{enumerate}
\item for every $C^{*}$-algebra $B,$ and for every surjection $\rho:B\rightarrow A,$
there is a sequence $\theta_{n}:A\rightarrow B$ of $*$-homomorphisms
so that $\rho\circ\theta_{n}(a)\rightarrow a$ for all $a$ in $A;$
\item there exists a projective $C^{*}$-algebra $P$ and surjection $\rho:P\rightarrow A$
for which there is a sequence $\theta_{n}:A\rightarrow P$ of $*$-homomorphisms
so that $\rho\circ\theta_{n}(a)\rightarrow a$ for all $a$ in $A.$ 
\end{enumerate}
\end{thm}

\begin{proof}
Suppose $A$ is weakly projective. Given $\rho:B\rightarrow A$ a
surjection, we can approximately lift the identity map on $A$ as
in this diagram: 
\[
\xymatrix{
	&B \ar[d]^{\rho}\\
A \ar@{=}[r] \ar@{-->}[ur] ^{\theta_n}
	& A
}  \]
We have proven (a), and it is
obvious that (a) implies (b).

Suppose we are given $\varphi:A\rightarrow C$ and 
a surjection $\pi:B\rightarrow C.$
Since $P$ is projective, we can find $\psi$ to make this diagram
commute:
\[
\xymatrix{
P \ar[d] ^{\rho} \ar[r] ^{\psi}
	& B \ar[d] ^{\pi}\\
A \ar[r] _{\varphi}
	& C
}  \]
The maps $\varphi_{n}=\psi\circ\theta_{n}$
show $A$ is weakly projective.
\end{proof}

\begin{thm}
Suppose $A$ is a separable $C^{*}$-algebra. Each of the following
two conditions is equivalent to $A$ being weakly
projective w.r.t.~unital $C^{*}$-algebras:

\begin{enumerate}
\item for every unital $C^{*}$-algebra $B,$ and for every unital surjection
$\rho:B\rightarrow\widetilde{A},$ there is a
sequence $\theta_{n}:\widetilde{A}\rightarrow B$
of unital $*$-homomorphisms so that $\rho\circ\theta_{n}(a)\rightarrow a$
for all $a$ in $\widetilde{A};$ 
\item there exists a projective $C^{*}$-algebra $P$ and a unital surjection
$\rho:\widetilde{P}\rightarrow\widetilde{A}$ for which there is a
sequence $\theta_{n}:\widetilde{A}\rightarrow\widetilde{P}$ of unital
$*$-homomorphisms so that $\rho\circ\theta_{n}(a)\rightarrow a$
for all $a$ in $\widetilde{A}.$ 
\end{enumerate}
\end{thm}

\begin{proof}
Suppose $A$ is weakly projective w.r.t.~unital $C^{*}$-algebras.
Given $\rho:B\rightarrow\widetilde{A}$ a unital surjection, we can
approximately lift the identity map on $\widetilde{A},$ as in this
diagram:
\[
\xymatrix{
	&B \ar[d]^{\rho}\\
{\widetilde {A}} \ar@{=}[r] \ar@{-->}[ur] ^{\theta_n}
	& {\widetilde {A}}
}
\]
We have proven (a). Again it is obvious
that (a) implies (b).

Suppose we are given $\varphi:A\rightarrow C$ and a unital surjection
$\pi:B\rightarrow C.$ Since $C$ is unital, we can extend $\varphi$
to a unital $*$-homomorphism
$\overline{\varphi}:\widetilde{A}\rightarrow C.$
Since $P$ is projective and $B$ is unital, we can find $\psi$ a
unital $*$-homomorphism to make this diagram commute:
\[
\xymatrix{
	&{\widetilde {P}} \ar[d] ^{\rho} \ar[r] ^{\psi}
		& B \ar[d] ^{\pi}\\
*+<0.4cm,0.1cm>{A} \ar@{^{(}->}[r] \ar@/_0.4cm/[rr] _{\varphi}
		&{\widetilde {A}}  \ar[r] ^{\overline {\varphi}}
		& C
}
\]
The
maps $\varphi_{n}=\psi\circ\theta_{n}$ show $A$ is weakly projective
w.r.t.~unital $C^{*}$-algebras.
\end{proof}

\section{Properties of Weakly Projective C*-algebras.}

\begin{defn}
A quotient $B=A/I$ of a separable $C^{*}$-algebra $A$ is an 
\emph{approximate retract of} $A$ if there is a
sequence $\lambda_{n}:B\rightarrow A$
of $*$-homomorphisms so that $\rho\circ\lambda_{n}(b)\rightarrow b$
for all $b$ in $B.$ Here $\rho$ is the canonical surjection.
\end{defn}

We we use WP to stand for weakly projective.

\begin{prop}
\label{pro:approxRetractOfWPisWP}
An approximate retract of a separable
WP $C^{*}$-algebra is WP.
\end{prop}

\begin{proof}
The proof is very similar to that of Theorem~\ref{thm:WPfromPapproximations}.
\end{proof}

We use RFD to stand for residually finite dimensional. Recall that
$A$ is RFD if $A$ has a separating family of finite dimensional
representations. To read about other properties equivalent to this,
see \cite{ArchboldRFD,ExelLoringRFDfreeProd}.

\begin{prop}
An approximate retract of a separable RFD $C^{*}$-algebra is RFD.
\end{prop}

\begin{proof}
Given nonzero $b$ in $B$ we may find $m$ so that 
$\lambda_{m}(b)\geq\frac{1}{2}\|b\|.$
Now take a finite dimensional representation of $A$ 
with $\pi(\lambda_{m}(b))\neq0.$
Then $\pi\circ\lambda_{m}$ is a finite dimensional representation
of $B$ that does not send $b$ to zero.
\end{proof}

\begin{thm}
A $C^{*}$-algebra that is weakly projective w.r.t.~unital $C^{*}$-algebras
is RFD.
\end{thm}

\begin{proof}
If $A$ is weakly projective w.r.t.~unital $C^{*}$-algebras it is
an approximate retract of the unitization of a projective $C^{*}$-algebra.
Projective $C^{*}$-algebras are 
RFD (\cite[Theorem 11.2.1]{Loring-lifting-perturbing}),
and therefore so are their unitizations.
\end{proof}

\begin{lem}
\label{lem:invariantsVanish}
If $A$ is weakly projective and $D$
is semiprojective then $[D,A]$ is trivial.
\end{lem}

\begin{proof}
Suppose $\varphi:D\rightarrow A$ is given. 
Let $\delta_{1}:\mathbf{C}A\rightarrow A$
be the map defined on the cone over $A$ by evaluation at $1.$ The
weak projectivity of $A$ provides us with $*$-homomorphisms 
$\psi_{n}:A\rightarrow\mathbf{C}A$
with $\delta_{1}\circ\psi_{n}\rightarrow\mbox{id}_{A}.$ 
Let $\varphi_{n}=\delta_{1}\circ\psi_{n}\circ\varphi$
so that $\varphi_{n}\sim0$ and $\varphi_{n}\rightarrow\varphi.$
By \cite[Theorem 3.6]{Blackadar-shape-theory} 
there is some $n$
for which $\varphi_{n}\sim\varphi.$ 
\end{proof}

\begin{thm}
If $A$ is weakly projective then $K_{*}(A)=0.$
\end{thm}

\section{Closure Properties}

The closure properties for projectivity 
found in \cite{LoringProjectiveCstar}
hold, and with practically the same proofs, for weak projectivity.
The proofs involve hereditary subalgebras generated by positive elements,
which are almost never unital, so we do not know about these closure
properties for weak projectivity w.r.t.\ unital $C^{*}$-algebras.

\begin{thm}
\label{thm:closedUnderMatrices}
If $A$ is separable and weakly projective
then $\mathbf{M}_{n}(A)$ is weakly projective.
\end{thm}

\begin{proof}
The proof is very similar to that 
of \cite[Theorem 10.2.3]{Loring-lifting-perturbing}.
\end{proof}

\begin{thm}
\label{thm:closedUnderSums} 
Suppose $A_{n}$ is separable for all
$n$ (finite or countable list). Then $\bigoplus_{n}A_{n}$ is weakly
projective if and only if each $A_{n}$ is weakly projective. 
\end{thm}

\begin{proof}
If the sum is WP, we use Proposition~\ref{pro:approxRetractOfWPisWP}
and the fact that summand is a retract of a direct sum to conclude
that each summand is WP.

For the converse, we have as 
in \cite[Theorem 10.1.13]{Loring-lifting-perturbing}
a way to lift orthogonal elements in the direct sum, each completely
positive in $A_{n},$ and so can reduce to a lifting problem of the
form
\[
\xymatrix{
  & {\bigoplus B_n }	\ar[d] ^{	\bigoplus \rho_n }\\
 {\bigoplus A_n }	\ar[r] _{	\bigoplus \varphi _n }&  {\bigoplus C_n }	 
}
\]
Suppose $F$ is a finite subset of $\bigoplus A_{n}$
with $F=\{a_{1},\ldots,a_{k}\}$ and 
$a_{j}=\left\langle a_{j,n}\right\rangle .$
There are $\psi_{n}:A_{n}\rightarrow C_{n}$ with 
\[
\left\Vert \rho_{n}\circ\psi_{n}(a_{j,n})-\varphi_{n}(a_{j,n})\right\Vert 
\leq\epsilon
\]
 for each $j.$ Then
 \[
\left\Vert 
\left(\bigoplus\rho_{n}\right)\circ\left(\bigoplus\psi_{n}\right)(a_{j})
-\left(\bigoplus\varphi_{n}\right)(a)
\right\Vert 
=\sup
\left\Vert
\rho_{n}\circ\psi_{n}(a_{j,n})-\varphi_{n}(a_{j,n})
\right\Vert 
\]
is also less than or equal to $\epsilon.$
\end{proof}

\section{Questions}

The Hilbert cube has nice properties, like local connectedness and
the fixed-point property, and these get inherited by all ARs and,
to a lesser extent, by all AARs. It would be nice to find similar
properties of a ``free'' $C^{*}$-algebra (generated by a universal
sequence of contractions).

\begin{question}

\label{pro:contractibleWPimpliesP?} 
Does contractability plus weak
projectivity imply projectivity? 

\end{question}

This question is motivated by the commutative situation. See \cite[Theorem 7.2]{ClappAANR}.
An answer may be hard to find, as Lemma~\ref{lem:invariantsVanish}
shows that all the obvious invariants vanish on the weak projectives.

\begin{question}

Is the class of $C^{*}$-algebras that are weakly projectivity w.r.t.\ unital
$C^{*}$-algebras closed under direct sums?

\end{question}

\begin{question}

Is the class of $C^{*}$-algebras that are weakly
projectivity w.r.t.\ unital $C^{*}$-algebras closed under
the formation of matrix algebras?

\end{question}

\begin{question}

For separable $C^{*}$-algebras, is it true that 
\begin{gather*}
\mathbf{M}_{2}(A)
\mbox{ is weakly projective w.r.t.\ unital }C^{*}\mbox{-algebras}\\
\qquad
\implies 
A
\mbox{ is weakly projective w.r.t.\ unital }C^{*}\mbox{-algebras?}
\end{gather*}

\end{question}

\begin{question}

For separable $C^{*}$-algebras, is it true that 
\[
\mathbf{M}_{2}(A)\mbox{ is weakly projective }
\implies 
A\mbox{ is weakly projective }?
\]

\end{question}

See \cite[Section 4]{BlackadarSPinSimple} 
and \cite[Section 3]{HadwinLiApproxLift}.


\begin{thebibliography}{10}

\bibitem{ArchboldRFD}
Robert~J. Archbold.
\newblock On residually finite-dimensional {$C\sp *$}-algebras.
\newblock {\em Proc. Amer. Math. Soc.}, 123(9):2935--2937, 1995.

\bibitem{Blackadar-shape-theory}
Bruce Blackadar.
\newblock Shape theory for {$C\sp \ast$}-algebras.
\newblock {\em Math. Scand.}, 56(2):249--275, 1985.

\bibitem{BlackadarSPinSimple}
Bruce Blackadar.
\newblock Semiprojectivity in simple {$C\sp *$}-algebras.
\newblock In {\em Operator algebras and applications}, volume~38 of {\em Adv.
  Stud. Pure Math.}, pages 1--17. Math. Soc. Japan, Tokyo, 2004.

\bibitem{BorsukTheoryOfShape}
Karol Borsuk.
\newblock {\em Theory of shape}.
\newblock PWN---Polish Scientific Publishers, Warsaw, 1975.
\newblock Monografie Matematyczne, Tom 59.

\bibitem{chigogidzeNonComARs}
Alex Chigogidze and Alexander~N. Dranishnikov.
\newblock Which compacta are noncommutative {AR}s?
\newblock {\em Topology Appl}, to appear.

\bibitem{ClappAANR}
Michael~H. Clapp.
\newblock On a generalization of absolute neighborhood retracts.
\newblock {\em Fund. Math.}, 70(2):117--130, 1971.

\bibitem{DadarlatElliottFieldsOfKirchberg}
Marius Dadarlat and George~A. Elliott.
\newblock One-parameter continuous fields of {K}irchberg algebras.
\newblock {\em Comm. Math. Phys.}, 274(3):795--819, 2007.

\bibitem{EffrosKaminkerShape}
Edward~G. Effros and Jerome~A. Kaminker.
\newblock Homotopy continuity and shape theory for {$C\sp \ast$}-algebras.
\newblock In {\em Geometric methods in operator algebras (Kyoto, 1983)}, volume
  123 of {\em Pitman Res. Notes Math. Ser.}, pages 152--180. Longman Sci.
  Tech., Harlow, 1986.

\bibitem{EilersLoringContingenciesStableRelations}
S{\o}ren Eilers and Terry~A. Loring.
\newblock Computing contingencies for stable relations.
\newblock {\em Internat. J. Math.}, 10(3):301--326, 1999.

\bibitem{ExelLoringRFDfreeProd}
Ruy Exel and Terry~A. Loring.
\newblock Finite-dimensional representations of free product {$C\sp
  *$}-algebras.
\newblock {\em Internat. J. Math.}, 3(4):469--476, 1992.

\bibitem{HadwinLiApproxLift}
Don Hadwin and Weihua Li.
\newblock A note on approximate liftings.
\newblock {\em Oper. Matrices}, 3(1):125--143, 2009.

\bibitem{HuRetracts}
Sze-tsen Hu.
\newblock {\em Theory of retracts}.
\newblock Wayne State University Press, Detroit, 1965.

\bibitem{LinWeakSemiprojectivity}
Huaxin Lin.
\newblock Weak semiprojectivity in purely infinite simple {$C\sp
  \ast$}-algebras.
\newblock {\em Canad. J. Math.}, 59(2):343--371, 2007.

\bibitem{LoringProjectiveCstar}
Terry~A. Loring.
\newblock Projective {$C\sp \ast$}-algebras.
\newblock {\em Math. Scand.}, 73(2):274--280, 1993.

\bibitem{Loring-lifting-perturbing}
Terry~A. Loring.
\newblock {\em Lifting solutions to perturbing problems in {$C\sp
  *$}-algebras}, volume~8 of {\em Fields Institute Monographs}.
\newblock American Mathematical Society, Providence, RI, 1997.

\bibitem{LoringProjectiveKtheory}
Terry~A. Loring.
\newblock A projective {$C\sp{\ast} $}-algebra related to {$K$}-theory.
\newblock {\em J. Funct. Anal.}, 254(12):3079--3092, 2008.

\bibitem{LorPederProjTrans}
Terry~A. Loring and Gert~K. Pedersen.
\newblock Projectivity, transitivity and {AF}-telescopes.
\newblock {\em Trans. Amer. Math. Soc.}, 350(11):4313--4339, 1998.

\bibitem{PedersenExtensions}
Gert~K. Pedersen.
\newblock Extensions of {$C\sp *$}-algebras.
\newblock In {\em Operator algebras and quantum field theory ({R}ome, 1996)},
  pages 1--35. Int. Press, Cambridge, MA, 1997.

\bibitem{PowersM_AANRs}
Michael~J. Powers.
\newblock Fixed point theorems for non-compact approximative {ANR}'s.
\newblock {\em Fund. Math.}, 75(1):61--68, 1972.

\bibitem{Shulman_nilpotents}
Tatiana Shulman.
\newblock {Lifting of nilpotent contractions}.
\newblock {\em Bulletin of the London Mathematical Society}, 40(6):1002, 2008.

\bibitem{SpielbergWeakSemiprojectivity}
Jack Spielberg.
\newblock Weak semiprojectivity for purely infinite {$C\sp \ast$}-algebras.
\newblock {\em Canad. Math. Bull.}, 50(3):460--468, 2007.

\bibitem{VanMillInfDimTop}
Jan van Mill.
\newblock {\em Infinite-dimensional topology}, volume~43 of {\em North-Holland
  Mathematical Library}.
\newblock North-Holland Publishing Co., Amsterdam, 1989.
\newblock Prerequisites and introduction.

\end{thebibliography}
\end{document}